\documentclass{amsart}

\usepackage{graphicx}
\usepackage{amssymb}
\usepackage{amsmath, amscd}
\usepackage{pinlabel}

\newtheorem{thm}{Theorem}[section]
\newtheorem{prp}[thm]{Proposition}
\newtheorem{cor}[thm]{Corollary}
\newtheorem{lma}[thm]{Lemma}

\theoremstyle{definition}

\theoremstyle{remark}

\newtheorem{rmk}[thm]{Remark}

\newcommand{\R}{{\mathbb{R}}}
\newcommand{\C}{{\mathbb{C}}}

\newcommand{\Z}{{\mathbb{Z}}}

\newcommand{\area}{\operatorname{area}}
\newcommand{\lk}{\operatorname{lk}}
\newcommand{\slk}{\operatorname{slk}}
\newcommand{\Aug}{\operatorname{Aug}}
\newcommand{\qAug}{\widehat{\operatorname{Aug}}}
\newcommand{\ind}{\operatorname{ind}}

\begin{document}

\title{Knot contact homology and open Gromov-Witten theory}
\author{Tobias Ekholm}
\address{Department of mathematics, Uppsala University, Box 480, 751 06 Uppsala, Sweden}
\email{tobias.ekholm\@@math.uu.se}
\address{Institut Mittag-Leffler, Aurav 17, 182 60 Djursholm, Sweden}
\email{ekholm\@@mittag-leffler.se}
\thanks{The author is supported by the Knut and Alice Wallenberg Foundation and by the Swedish Research Council.}


\maketitle

\begin{abstract}
Knot contact homology studies symplectic and contact geometric properties of conormals of knots in 3-manifolds using holomorphic curve techniques. It has connections to both mathematical and physical theories. On the mathematical side, we review the theory, show that it gives a complete knot invariant, and discuss its connections to Fukaya categories, string topology, and micro-local sheaves. On the physical side, we describe the connection between the augmentation variety of knot contact homology and Gromov-Witten disk potentials, and discuss the corresponding higher genus relation that quantizes the augmentation variety.
\end{abstract}

\section{Introduction}\label{Sec:intr}
If $M$ is an oriented 3-manifold then its 6-dimensional cotangent bundle $T^{\ast}M$ with the closed non-degenerate 2-form $\omega=-d\theta$, where $\theta=pdq$ is the Liouville or action 1-form, is a symplectic manifold. As a symplectic manifold, $T^{\ast}M$ satisfies the Calabi-Yau condition, $c_{1}(T^{\ast}M)=0$, and is thus a natural ambient space for the topological string theory of physics and its mathematical counterpart, Gromov-Witten theory. 

If $K\subset M$ is a knot then its \emph{Lagrangian conormal} $L_{K}\subset T^{\ast}M$ of covectors along $K$ that annihilate the tangent vector of $K$ is a Lagrangian submanifold (i.e., $\omega|_{L_{K}}=0$) diffeomorphic to $S^{1}\times\R^{2}$. Lagrangian submanifolds provide natural boundary conditions for open string theory or open Gromov-Witten theory, that counts holomorphic curves with boundary on the Lagrangian. 

Here we will approach the Gromov-Witten theory of $L_{K}$ from geometric data at infinity. At infinity, the pair $(T^{\ast}M, L_{K})$ has ideal contact boundary $(ST^{\ast}M,\Lambda_{K})$, the unit sphere cotangent bundle $ST^{\ast}M$ with the contact form $\alpha=\theta|_{ST^{\ast}M}$ and $\Lambda_{K}$ the \emph{Legendrian conormal} ($\alpha|_{\Lambda_{K}}=0$) $\Lambda_{K}=L_{K}\cap ST^{\ast}M$. In what follows we will restrict attention to the most basic cases of knots in 3-space or the 3-sphere, $M=\R^{3}$ or $M=S^{3}$.

\subsection{Mathematical aspects of knot contact homology}\label{sec:math}
There is a variety of holomorphic curve theories, all interconnected, that can be applied to distinguish objects up to deformation in contact and symplectic geometry. Knot contact homology belongs to a framework of such theories called Symplectic Field Theory (SFT) \cite{EGH}. More precisely, it is the most basic version of SFT, the Chekanov-Eliashberg dg-algebra $CE(\Lambda_{K})$, of the Legendrian conormal torus $\Lambda_{K}\subset ST^{\ast} \R^{3}$ of a knot $K\subset \R^{3}$. The study of knot contact homology was initiated by Eliashberg, see \cite{yasha}, around 2000 and developed from a combinatorial perspective by Ng \cite{Ngframed,Ngtransverse} and with holomorphic curve techniques in \cite{EENS,EENStransverse}.

Our first result states that the contact deformation class of $\Lambda_{K}$ encodes the isotopy class of $K$. Let $p\in \R^{3}$ be a point not on $K$ and let $\Lambda_{p}\subset ST^{\ast}\R^{3}$ denote the Legendrian conormal sphere of $p$. We consider certain filtered quotients of $CE(\Lambda_{K}\cup\Lambda_{p})$, called $R_{Kp}$, $R_{pK}$, and $R_{KK}$, together with a product operation $m\colon R_{Kp}\otimes R_{pK}\to R_{KK}$, borrowed from wrapped Floer cohomology. 

\begin{thm}\label{t:complete}
\cite[Theorem 1.1]{ENS}	
Two knots $K,J\subset \R^{3}$ are isotopic if and only if the triples $(R_{Kp},R_{pK},R_{KK})$ and $(R_{Jp},R_{pJ},R_{JJ})$, with the product $m$, are quasi-isomorphic. It follows in particular that $\Lambda_{K}$ and $\Lambda_{J}$ are (parameterized) Legendrian isotopic if and only if $K$ and $J$ are isotopic.
\end{thm}
A version of this theorem was first proved by Shende \cite{shende} using micro-local sheaves and was reproved using holomorphic disks in \cite{ENS}. We point out that the Legendrian conormal tori of any two knots are smoothly isotopic when considered as ordinary submanifolds of $ST^{\ast}\R^{3}$. Theorem \ref{t:complete} and its relations to string topology, Floer cohomology, and micro-local sheaves are discussed in Section \ref{Sec:complete}.

\subsection{Physical aspects of knot contact homology}\label{sec:physics}
We start from Witten's relation between Chern-Simons gauge theory and open topological string \cite{Witten:1992fb} together with Ooguri-Vafa's study of large $N$ duality for conormals of knots \cite{OV,Ooguri_Vafa_worldsheet}. Let $M$ be a closed 3-manifold. Witten identified the partition function of $U(N)$ Chern-Simons gauge theory on $M$ with the partition function of open topological string on $T^{\ast} M$ with $N$ branes on the Lagrangian zero-section $M$. In Chern-Simons theory, the $n$-colored HOMFLY-PT polynomial of a knot $K\subset M$ equals the expectation value of the holonomy around the knot of the $U(N)$-connection in the $n^{\rm th}$ symmetric representation. The generating function of $n$-colored HOMFLY-PT polynomials correspond on the string side to the partition function of open string theory in $T^{\ast}M$ with $N$ branes on $M$ and one brane on the conormal $L_{K}$ of the knot. 

For $M=S^{3}$, large $N$ duality says that the open string in $T^{\ast} S^{3}$ with $N$-branes on $S^{3}$ is equivalent to the closed string, or Gromov-Witten theory, in the non-compact Calabi-Yau manifold $X$ which is the total space of the bundle $\mathcal{O}(-1)^{\oplus 2}\to\C P^{1}$ (the resolved conifold), provided $\area(\C P^{1})=N g_{s}$, where $g_{s}$ is the string coupling, or genus, parameter.  As smooth manifolds, $X-\C P^{1}$ and $T^{\ast} S^{3}- S^{3}$ are diffeomorphic. As symplectic manifolds they are closely related, in particular both are asymptotic to $[0,\infty)\times ST^{\ast}S^{3}$ at infinity.

If $K\subset S^{3}$ is a knot then after a non-exact shift, see \cite{koshkin}, $L_{K}\subset T^{\ast} S^{3}-S^{3}$, and we can view $L_{K}$ as a Lagrangian submanifold in $X$. This leads to the following relation between the colored HOMFLY-PT polynomial and open topological string or open Gromov-Witten theory in $X$. Let $C_{\chi,r,n}$ be the count of (generalized) holomorphic curves in $X$ with boundary on $L_{K}$, of Euler characteristic $\chi$, in relative homology class $rt+ nx$, where $t$ is the class of $[\C P^{1}]\in H_{2}(X,L_{K})$ and $x\in H_{2}(X,L_{K})$ maps to the generator of $H_{1}(L_{K})$ under the connecting homomorphism. If  
\[ 
F_{K}(e^{x},g_{s},Q)=\sum_{n,r,\chi} C_{n,r,\chi} \,g_{s}^{-\chi}Q^{r}e^{nx},
\]
then
\[ 
\Psi_{K}(x):=e^{F_{K}(x)} = \sum H_{K,n}(q, Q) e^{nx}, \quad q=e^{g_{s}}, \; Q=q^{N},
\] 
where $H_{K,n}$ denotes the $n$-colored HOMFLY-PT polynomial of $K$.

The colored HOMFLY-PT polynomial is $q$-holonomic \cite{Garoufalidis}, which in our language can be expressed as follows. Let $e^{\hat x}$ denote the operator which is multiplication by $e^{x}$ and $e^{\hat p}=e^{g_{s}\frac{\partial}{\partial x}}$. Then there is a polynomial $\hat A_{K}=\hat A_{K}(e^{\hat x},e^{\hat p})$ such that $\hat A_{K}\Psi_{K}=0$. 

We view $Q$ as a parameter and think of it as fixed. Then from the short-wave asymptotic expansion of the wave function $\Psi_{K}$,
\[ 
\Psi_{K}(x)=e^{F_{K}}=\exp\left(g_{s}^{-1}W_{K}^{0}(x)+ W_{K}^{1}(x) + g_{s}^{j-1}W_{K}^{j}(x) +\dots\right),
\]
we find that $p=\frac{\partial W_{K}^{0}}{\partial x}$ parameterizes the algebraic curve $\{A_{K}(e^{x},e^{p})=0\}$, where the polynomial $A_{K}$ is the classical limit $g_{s}\to 0$ of the operator polynomial $\hat A_{K}$. In terms of Gromov-Witten theory, $W_{K}(x)=W_{K}^{0}(x)$ can be interpreted as the disk potential, the count of holomorhic disks ($\chi=1$ curves) in $X$ with boundary on $L_{K}$.

In \cite{AV} it was observed (in computed examples) that the polynomial $A_{K}$ agreed with the \emph{augmentation polynomial} $\Aug_{K}$ of knot contact homology. To describe that polynomial, we consider a version $\mathcal{A}_{K}$ of $CE(\Lambda_{K})$ with coefficients in the group algebra of the second relative homology $\C[H_{2}(ST^{\ast} S^{3},\Lambda_{K})]\approx \C[e^{\pm x},e^{\pm p},Q^{\pm 1}]$, where $x$ and $p$ map to the longitude and meridian generators of $H_{1}(\Lambda_{K})$, and $Q=e^{t}$ for $t=[ST_{p}^{\ast}S^{3}]$, the class of the fiber sphere. If $\C$ is considered as a dg-algebra in degree 0 then the \emph{augmentation variety} $V_{K}$ is the closure of the set in the space of coefficients where there is a chain map into $\C$:  
\[ 
V_{K}=\text{closure}\bigl(\bigl\{(e^{x},e^{p},Q)\colon \text{there exists a chain map } \epsilon\colon \mathcal{A}_{K}\to\C\bigr\}\bigr),
\]
and the \emph{augmentation polynomial} $\Aug_{K}$ is its defining polynomial. We have the following result that  connects knot contact homology and Gromov-Witten theory at the level of the disk.
\begin{thm}\label{t:diskpotential}
\cite[Theorem 6.6 and Remark 6.7]{AENV}	
If $W_{K}(x)$ is the Gromov-Witten disk potential of $L_{K}\subset X$ then $p=\frac{\partial W_{K}}{\partial x}$ parameterizes a branch of the augmentation variety $V_{K}$.
\end{thm} 
The augmentation polynomial $\Aug_{K}$ of a knot $K$ is obtained by elimination theory from explicit polynomial equations. Theorem \ref{t:diskpotential} thus leads to a rather effective indirect calculation of the Gromov-Witten disk potential. It is explained in Section \ref{Sec:aug+disk}.   

In Section \ref{Sec:SFT} we discuss the higher genus counterpart of Theorem \ref{t:diskpotential}. We sketch the construction of a higher genus generalization of knot contact homology that we call \emph{Legendrian SFT}. In this theory, the operators $e^{\hat x}$ and $e^{\hat p}$ have natural enumerative geometrical interpretations. Furthermore, in analogy with the calculation of the augmentation polynomial,  elimination theory in the non-commutative setting should give the operator polynomial $\qAug_{K}(e^{\hat x},e^{\hat p})$ such that $\qAug_{K}\Psi_{K}=0$, and thus determine the recursion relation for the colored HOMFLY-PT.

\begin{rmk}
Theorem \ref{t:diskpotential} and other results about open Gromov-Witten theory presented here should be considered established from the physics point of view. From a more strict mathematical perspective, they are not rigorously proved and should be considered as conjectures.  	
\end{rmk}	

\subsection*{Acknowledgements} I am much indebted to my coauthors, Aganagic, Cieliebak, Etnyre, Latchev, Lekili, Ng, Shende, Sullivan, and Vafa, of the papers on which this note is based.

\section{Knot contact homology and Chekanov-Eliashberg dg-algebras}
In this section we introduce Chekanov-Eliashberg dg-algebras in the cases we use them.

\subsection{Background notions}
Let $M$ be an orientable 3-manifold and consider the unit cotangent bundle $ST^{\ast}M$ with the contact 1-form $\alpha$ which is the restriction of the action form $pdq$. The hyperplane field $\xi=\ker(\alpha)$ is the contact structure determined by $\alpha$ and $d\alpha$ gives a symplectic form on $\xi$. The first Chern-class of $\xi$ vanishes, $c_{1}(\xi)=0$. 

Let $\Lambda\subset ST^{\ast} M$ be a Legendrian submanifold, $\alpha|_{\Lambda}=0$. Then the tangent spaces of $\Lambda$ are Lagrangian subspaces of $\xi$. Since $c_{1}(\xi)=0$ there is a Maslov class in $H^{1}(\Lambda;\Z)$ that measures the total rotation of $T\Lambda$ in $\xi$. Here we will consider only Legendrian submanifolds with vanishing Maslov class.

The \emph{Reeb vector field} $R$ of $\alpha$ is characterized by $d\alpha(\cdot,R)=0$ and $\alpha(R)=1$. Flow segments of $R$ that begin and end on $\Lambda$ are called \emph{Reeb chords}.
The Reeb flow on $ST^{\ast}M$ is the lift of the geodesic flow on $M$. Consequently, if $K\subset M$ is a knot (or any submanifold) then Reeb chords of $\Lambda_{K}$ correspond to geodesics connecting $K$ to itself and perpendicular to $K$ at its endpoints. 

\subsection{Coefficients in chains on the based loop space}\label{sec:loopspace}
Let $M=\R^{3}$, $K\subset \R^{3}$ be a knot and $p\in \R^{3}-K$ a point. Let $\Lambda_{0}=\Lambda_{p}$, $\Lambda_{1}=\Lambda_{K}$, and $\Lambda=\Lambda_{0}\cup\Lambda_{1}$. The algebra $CE(\Lambda)$ is generated by the Reeb chords of $\Lambda$ and homotopy classes of loops in $\Lambda$. 
We define the coefficient ring $\mathbf{k}_{\Lambda}$ as the algebra over $\C$ generated by idempotents $e_{j}$ corresponding to $\Lambda_{j}$ so that 
$e_{i}e_{j}=\delta_{ij}e_{i}$, $i,j\in\{0,1\}$,  
where $\delta_{ij}$ is the Kronecker delta. 

Note that $\Lambda_{0}$ is a sphere and $\Lambda_{1}$ is a torus. Fix generators $\lambda$ and $\mu$ of $\pi_{1}(\Lambda_{1})$ (corresponding to the longitude and the meridian of $K$) and think of them as generators of the group algebra $\C[\pi_{1}(\Lambda_{1})]\approx \C[\lambda^{\pm 1},\mu^{\pm 1}]$. We let $CE(\Lambda)$ be the algebra over $\mathbf{k}_{\Lambda}$ generated by Reeb chords $c$, and the homotopy classes $\lambda$ and $\mu$. The generators $\lambda$ and $\mu$ satisfy the relations in the group algebra and the following additional relations hold:
\begin{align*} 
ce_{j}=
\begin{cases}
c &\text{if $c$ starts on $\Lambda_{j}$},\\
0 &\text{otherwise},
\end{cases}
\quad & \quad
e_{k}c=
\begin{cases}
c &\text{if $c$ ends on $\Lambda_{k}$},\\
0 &\text{otherwise},
\end{cases}\\
e_{j}\lambda_{k}=\lambda_{k}e_{j}=\delta_{jk}\lambda_{k},\quad &\quad
e_{j}\mu_{k}=\mu_{k}e_{j}=\delta_{jk}\mu_{k}.
\end{align*}

The grading of $\lambda$ and $\mu$ is $|\lambda|=|\mu|=0$ and Reeb chords are graded by the Conley-Zehnder index, which in the case of knot contact homology equals the Morse index of the underlying binormal geodesic, see \cite{EENS}. We can thus think of elements of $CE(\Lambda)$ as finite linear combinations of composable monomials $\mathbf{c}$ of the form
\[ 
\mathbf{c}=\gamma_{0}c_{1}\gamma_{1}c_{2}\gamma_{2}\dots\gamma_{m-1}c_{m}\gamma_{m},
\] 
where $\gamma_{j}$ is a homotopy class of loops in $\Lambda$ and $c_{j+1}$ is a Reeb chord, and composable means that $c$ starts at the component of $\gamma_{j}$ and ends at the component of $\gamma_{j-1}$. We then have the decomposition
\[ 
CE(\Lambda)=\bigoplus_{i,j} CE(\Lambda)_{i,j},
\]
where $CE(\Lambda)_{i,j}$ is generated by monomials which start on $\Lambda_{j}$ and ends on $\Lambda_{i}$. The product of two monomials is given by concatenation if the result is composable and zero otherwise.

The differential is defined to be $0$ on $e_i$ and on elements of $\Z[\pi_1(\Lambda_1)]$
and is given by a holomorphic disk count on Reeb chord generators that we describe next. Fix a complex structure $J$ on the symplectization $\R\times ST^{\ast}\R^{3}$, with symplectic form $d(e^{t}\alpha)$, $t\in\R$, that is invariant under the $\R$-translation and maps $\xi$ to itself. If $c$ is a Reeb chord then $\R\times c$ is a holomorphic strip with boundary on the Lagrangian submanifold $\R\times\Lambda$. Fix a base point in each component of $\Lambda$ and fix for each Reeb chord endpoint a reference path connecting it to the base point. Consider a Reeb chord $a$ and a composable word $\mathbf{b}$ of homotopy classes and Reeb chords of the form 
\[ 
\mathbf{b}=\gamma_{0}b_{1}\gamma_{1}b_{2}\gamma_{2}\dots\gamma_{m-1}b_{m}\gamma_{m},
\]
where $\gamma_{0}$ lies in the component where $a$ ends and $\gamma_{m}$ in the component where $a$ starts. We let $\mathcal{M}(a;\mathbf{b})$ denote the moduli space of holomorphic disks 
\[ 
u\colon (D,\partial D)\to (\R\times ST^{\ast}\R^{3},\R\times\Lambda), \quad du + J\circ du\circ i=0,
\]
with one positive and $m$ negative boundary punctures, 
which are asymptotic to the Reeb chord strip $\R\times a$ at positive infinity at the positive puncture and to the Reeb chord strip $\R\times b_{j}$ at negative infinty at the $j^{\rm th}$ negative puncture and such that the closed off path between punctures $j$ and $j+1$ lies in homotopy class $\gamma_{j}$, where puncture $0$ and $m+1$ both refer to the positive puncture, see Figure \ref{fig:differential}. The dimension of the moduli space $\mathcal{M}(a;\mathbf{b})$ equals $|a|-|\mathbf{b}|$. 

We define
\begin{equation}\label{eq:differential} 
\partial a = \sum_{a-|\mathbf{b}|=1}|\mathcal{M}(a;\mathbf{b})|\mathbf{b},
\end{equation}
where $|\mathcal{M}(a;\mathbf{b})|$ denotes the algebraic number of $\R$-families of disks in $\mathcal{M}(a;\mathbf{b})$ and extend to monomials by Leibniz rule. For the count in \eqref{eq:differential} to make sense we need the solutions to be transversely cut out. Since disks with one positive puncture cannot be multiple covers, transversality is relatively straightforward. Furthermore, the sum is finite by the SFT version of Gromov compactness.

\begin{figure}[htp]
	\labellist
	\small\hair 2pt
	\pinlabel $b_{1}$ at 180 73
	\pinlabel $b_{2}$ at 333 73
	\pinlabel $b_{3}$ at 480 73
	\pinlabel $\gamma_{0}$ at 190 500
	\pinlabel $\gamma_{1}$ at 250 270
	\pinlabel $\gamma_{2}$ at 420 270
	\pinlabel $\gamma_{3}$ at 450 500
	\pinlabel $a$ at 330 775
	\endlabellist
	\centering
	\includegraphics[width=.3\linewidth]{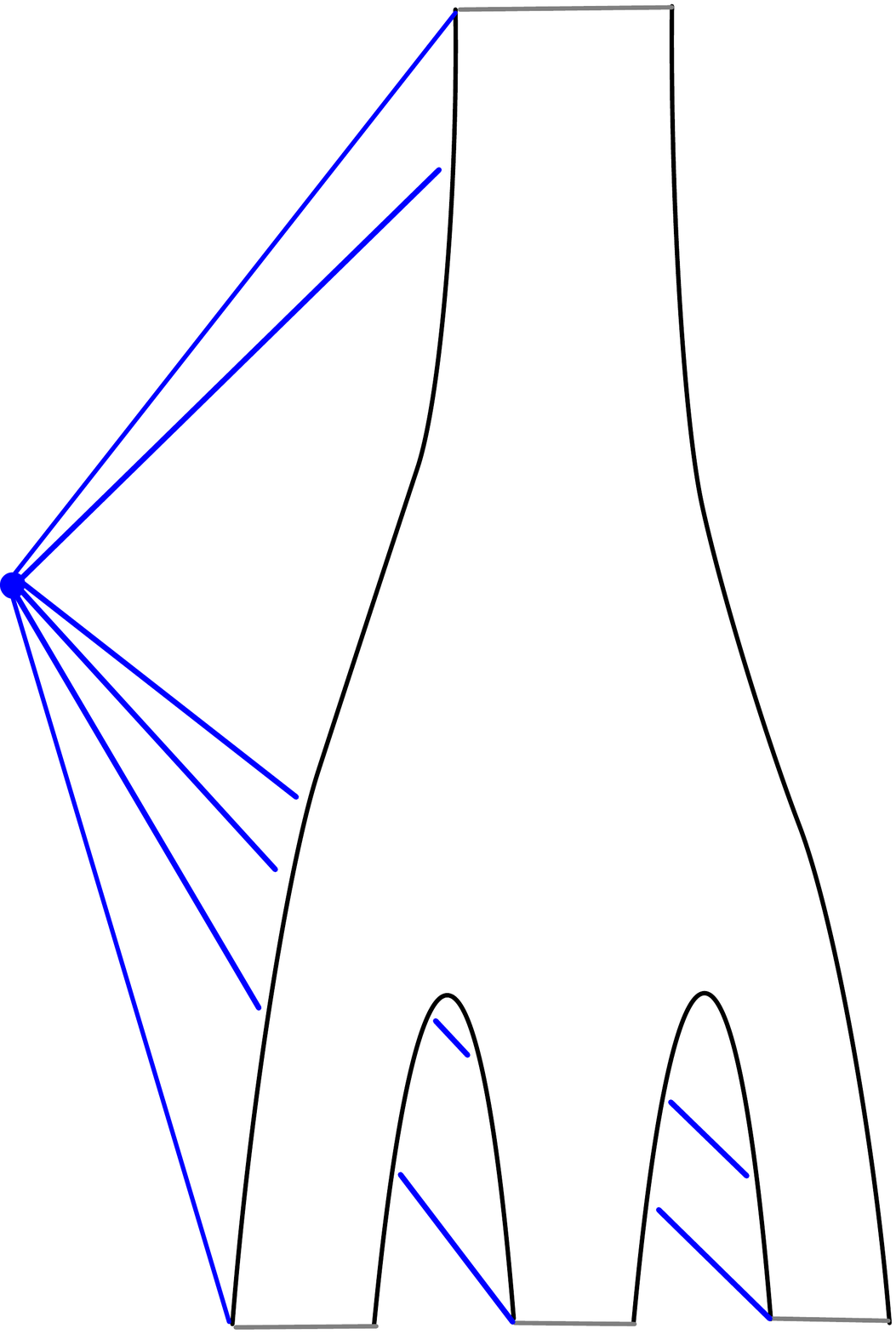}
	\caption{A disk contributing $\gamma_{0}b_{1}\gamma_{1}b_{2}\gamma_{2}b_{3}\gamma_{3}$ to $\partial a$.}
	\label{fig:differential}
\end{figure}

The basic result for Chekanov-Elisahberg algebras is then the following.

\begin{lma}\label{l:dgabasics}
The map $\partial$ is a differential, $\partial\circ\partial=0$ and the quasi-isomorphism class of $CE(\Lambda)$ is invariant under Legnedrian isotopies of $\Lambda$. Furthermore, the differential respects the decomposition $CE(\Lambda)=\bigoplus_{i,j} CE(\Lambda)_{i,j}$ which thus descends to homology.
\end{lma}

\begin{rmk}
For general contact manifolds, $CE(\Lambda)$ is an algebra over the so called orbit contact homology algebra. In the cases under study, $ST^{\ast}\R^{3}$ and $ST^{\ast}S^{3}$, the orbit contact homology algebra is trivial in degree $0$ and can be neglected. 
\end{rmk}

\begin{rmk}
For general Legendrian submanifolds $\Lambda$, the version of $CE(\Lambda)$ considered here is more complicated. The group ring generators for torus components are replaced by chains on the based loop space of the corresponding components and moduli spaces of all dimensions contribute to the differential, see \cite{EL}. 
\end{rmk}

\begin{proof}[Sketch of proof]
If $a$ is a Reeb chord then $\partial(\partial a)$ counts two level curves joined at Reeb chords. By gluing and SFT compactness such configurations constitute the boundary of an oriented 1-manifold and hence cancel algebraically. The invariance property can be proved in a similar way by looking at the boundary of the moduli space of holomorphic disks in Lagrangian cobordisms associated to Legendrian isotopies. See e.g.~\cite{EES} for details. 
\end{proof}

\subsection{Coefficients in relative homology}
Our second version of the Chekanov-Eliashberg dg-algebra of the conormal $\Lambda_{K}\subset ST^{\ast}S^{3}$ of a knot $K\subset S^{3}$ is denoted $\mathcal{A}_{K}$. The algebra $\mathcal{A}_{K}$ is generated by Reeb chords graded as before. Its coefficient ring is the group algebra $\C[H_{2}(ST^{\ast} S^{3},\Lambda_{K})]$ and group algebra elements commute with Reeb chords. To define the differential we fix for each Reeb chord a disk filling the reference paths. Capping off punctured disks in the moduli space $\mathcal{M}(a,\mathbf{b})$ with these disks we get a relative homology class and define the differential on Reeb chord generators of $\mathcal{A}_{K}$ as
\[ 
d a = \sum_{|a|-|\mathbf{c}|=1}|\mathcal{M}(a;\mathbf{c})|\mathbf{c}.
\]
Here $\mathbf{c}=e^{A}c_{1}\dots c_{m}$, where $c_{j}$ are the Reeb chords at the negative punctures of the disks in the moduli space and $A\in H_{2}(ST^{\ast}S^{3};\Lambda_{K})$ is the relative homology class of the capped off disks. That $d$ is a differential and the quasi-isomorphism invariance of $\mathcal{A}_{K}$ under Legendrian isotopies follow as before.

\subsection{Knot contact homology in basic examples}\label{sec:ex1}
We calculate the knot contact homology dg-algebras (in the lowest degrees) for the unknot and the trefoil knot. For general formulas we refer to \cite{EENS,EENStransverse}. The expressions give the differential in $\mathcal{A}_{K}$. for the differential in $CE(\Lambda_{K})$, set $Q=1$, $e^{x}=\lambda$, and $e^{p}=\mu$.

\subsubsection{The unknot}
Representing the unkot as a round circle in the plane we find that it has an $S^{1}$-Bott family of binormal geodesics and correspondingly an $S^{1}$-Bott family of Reeb chords. After small perturbation this gives two Reeb chords $c$ and $e$ of degrees $|c|=1$ and $|e|=2$. The differential can be computed using Morse flow trees, see \cite{E,EENS}. The result is
\begin{equation}\label{eq:unknotdiff}
de=0,\quad dc=1-e^{x}-e^{p}-Qe^{x}e^{p}.
\end{equation}

\subsubsection{The trefoil knot}\label{ssec:trefoildiff}
Represent the trefoil knot as a 2-strand braid around the unkot. If the trefoil $T$ lies sufficiently close to the unkot $U$, then its conormal torus $\Lambda_{T}$ lies in a small neighborhood $N(\Lambda_{U})$ of the unknot conormal, which can be identified with the neighborhood the zero section in its 1-jet space $J^{1}(\Lambda_{U})$. The projection $\Lambda_{T}\to\Lambda_{U}$ is a 2-fold cover and holomorphic disks with boundary on $\R\times\Lambda_{T}$ correspond to holomorphic disks on $\Lambda_{U}$ with flow trees attached, where the flow trees are determined by $\Lambda_{T}\subset J^{1}(\Lambda_{U})$, see \cite{EENS}. 
This leads to the following description of $\mathcal{A}_{T}$ in degrees $\le 1$. The Reeb chords are: 
\[ 
\text{degree 1: }b_{12},\, b_{21},\, c_{11},\, c_{12},\, c_{21},\, c_{22},\quad 
\text{degree 0: }a_{12},\, a_{21},
\]
with differentials
\begin{alignat*}{2}
&d c_{11} = e^{x}e^{p} - e^{x} -(2Q-e^{p})a_{12}-Qa_{12}^{2}a_{21},\quad 
&&d c_{12} = Q - e^{p} +e^{p} a_{12} + Q a_{12}a_{21},\\\notag
&d c_{21} = Q - e^{p} - e^{x}e^{p}a_{21} + Q a_{12}a_{21},\quad 
&&d c_{22}= e^{p} - 1 -Qa_{21}+e^{p} a_{12}a_{21},\\
&db_{12} = e^{-x}a_{12}-a_{21},\quad 
&&d b_{21}= a_{21}-e^{x}a_{12}.
\end{alignat*}

\section{A complete knot invariant}\label{Sec:complete}
In this section we discuss the completeness of knot contact homology as a knot invariant and describe its relations of to string topology, wrapped Floer cohomology, and micro-local sheaves. 

\subsection{Filtered quotients and a product}
We use notation as in Section \ref{sec:loopspace}, $\Lambda=\Lambda_{p}\cup\Lambda_{K}=\Lambda_{0}\cup\Lambda_{1}$, and consider $CE(\Lambda)$. The group ring $\Z[\pi_{1}(\Lambda_{K})]$ is a subalgebra of $CE(\Lambda)$ generated by the longitude and meridian generators $\lambda^{\pm 1}$ and $\mu^{\pm 1}$. Other generators are Reeb chords that correspond to binormal geodesics. If $\gamma$ is a geodesic we write $c$ for the corresponding Reeb chord. The grading of Reeb chords with endpoints on the same connected component is well-defined, while the grading for \emph{mixed chords} connecting distinct components are defined only up to an over all shift specified by a cetrain reference path connecting the two components. Let $\ind(\gamma)$ denote the Morse index of the geodesic $\gamma$.   

\begin{lma}\label{l:grading}\cite[Proposition 2.3]{ENS}
There is a choice of reference path so that the grading in $CE(\Lambda)$ of a Reeb chord $c$ corresponding to the geodesic $\gamma$ is as follows:
if $c$ connects $\Lambda_{K}$ to $\Lambda_{K}$ or $\Lambda_{p}$ to $\Lambda_{K}$ then
$|c|=\ind(\gamma)$, and if $c$ connects $\Lambda_{K}$ to $\Lambda_{p}$ then
$|c|=\ind(\gamma)+1$. 
\end{lma}
Consider the filtration on $CE(\Lambda)$ by the number of mixed Reeb chords, and the corresponding filtered quotients:
\begin{align*}
&CE(\Lambda)_{1,1}=\mathcal{F}^{0}_{11}\supset \mathcal{F}^{2}_{11}\supset \mathcal{F}^{4}_{11}\supset\dots,\quad CE_{11}^{(2k)}=\mathcal{F}_{11}^{2k}/\mathcal{F}_{11}^{2k+2},\\
&CE(\Lambda)_{i,j}=\mathcal{F}^{1}_{ij}\supset \mathcal{F}^{3}_{ij}\supset \mathcal{F}^{5}_{ij}\supset\dots,\quad CE_{ij}^{(2k+1)}=\mathcal{F}_{ij}^{2k+1}/\mathcal{F}_{ij}^{2k+3},\text{ for }i\ne j,
\end{align*}
where $\mathcal{F}^{r}$ denotes the subalgebra generated by monomials with at least $r$ mixed Reeb chords. The differential respects this filtration. Lemma \ref{l:grading} shows that $CE(\Lambda)$ is supported in non-negative degrees and that monomials of lowest degree $d(i,j)\in\{0,1\}$ in $CE(\Lambda)_{i,j}$ contain the minimal possible number $s(i,j)\in\{0,1\}$ of mixed Reeb chords. We then find that  
$H_{d(i,j)}\left(CE(\Lambda)_{i,j}\right)=H_{d(i,j)}(CE_{ij}^{(s(i,j))})$.
We call
\[ 
(R_{Kp},R_{pK},R_{KK}):=
\left(H_{0}(CE(\Lambda)_{10}),H_{1}(CE(\Lambda)_{01}),H_{0}(CE(\Lambda)_{1,1})\right)
\] 
the \emph{knot contact homology triple} of $K$. The concatenation product in $CE(\Lambda)$ turns $R_{Kp}$ and $R_{pK}$ into left and right modules, respectively, and $R_{KK}$ into a left-right module over $\Z[\lambda^{\pm 1},\mu^{\pm 1}]$.

We next consider a product for the knot contact homology triple that is closely related to the product in wrapped Floer cohomology. As the differential, it is defined in terms of moduli spaces of holomorphic disk but for the product there are two positive punctures rather than one. 

Let $a$ and $b$ be Reeb chords connecting $\Lambda_{p}$ to $\Lambda_{K}$ and vice versa. Let $\mathbf{c}$ be a monomial in $CE(\Lambda_{K})$. Define $\mathcal{M}(a,b;\mathbf{c})$ as the moduli space of holomorphic disks $u\colon D\to \R\times T^{\ast}\R^{3}$ with two positive punctures asymptotic to $a$ and $b$, such that the boundary arc between them maps to $\R\times\Lambda_{p}$, and such that the remaining punctured arc in the boundary maps to $\Lambda_{K}$ with homotopy class and negative punctures according to $\mathbf{c}$. We then have
\[ 
\dim(\mathcal{M}(a,b;\mathbf{c}))=|a|+|b|-|\mathbf{c}|.
\]
Define
\[ 
m'(a,b)=\sum_{|\mathbf{c}|=|a|+|b|-1}|\mathcal{M}(a,b;\mathbf{c})|\mathbf{c}
\]  
and use this to define the chain level product 
$
m\colon CE(\Lambda)_{10}^{(1)}\otimes CE(\Lambda)_{01}^{(1)} \to CE(\Lambda)_{11}^{(0)}
$
as 
$ 
m(\mathbf{a}a,b\mathbf{b})=\mathbf{a}m'(a,b)\mathbf{b}.
$
\begin{prp}\cite[Proposition 2.13]{ENS}
The product $m$ descends to homology and gives a product $m\colon R_{Kp}\otimes R_{pK}\to R_{KK}$. The knot contact homology triple as modules over $\Z[\pi_{1}(\Lambda_{K})]$ and with the product $m$ is invariant under Legendrian isotopy.
\end{prp}

\subsection{String topology and the cord algebra}
In this section we define a topological model for knot contact homology in low degrees that one can think of as the string topology of a certain singular space. Our treatment will be brief and we refer to \cite{CELN,ENS} for full details.

Let $K\subset \R^{3}$ be a knot and $p\in \R^{3}-K$ a point with Lagrangian conormals $L_{K}$ and $L_{p}$. Let $\Sigma$ be the union $\Sigma=\R^{3}\cup L_{K}\cup L_{p}\subset T^{\ast} \R^{3}$. Pick an almost complex structure $J$ compatible with the metric along the zero section. Fix base points $x_{K}\in L_{K}-\R^{3}$ and $x_{p}\in L_{p}-\R^{3}$. 

We consider broken strings which are paths $s\colon [a,b]\to \Sigma$ that connect base points, $c(a), c(b)\in \{x_{p},x_{K}\}$ and that admit a subdivision $a< t_{1}<\dots< t_{m}< b$ such that $s|_{[t_{i},t_{i+1}]}$ is a $C^{k}$-map into one of the irreducible components of $\Sigma$ and such that the left and right derivatives at switches (i.e., points where $c$ switches irreducible components) are related by $\dot c(t_{j}-) = J\dot c(t_{j}+)$.

For $\ell>0$, let $\Sigma_{\ell}$ denote the space of strings with $\ell$ switches at $p$ and with the $C^{k}$-topology for some $k>0$. Write $\Sigma_{\ell}=\Sigma_{\ell}^{KK}\cup \Sigma_{\ell}^{Kp}\cup\Sigma_{\ell}^{pK}\cup\Sigma_{\ell}^{pp}$,
where $\Sigma_{\ell}^{KK}$ denotes strings that start and end at $x_{K}$, etc. For $d>0$, let 
\[ 
C_{d}(\Sigma_{\ell})=C_{d}(\Sigma_{\ell}^{KK})\oplus C_{d}(\Sigma_{\ell}^{Kp}) \oplus C_{d}(\Sigma_{\ell}^{pK})\oplus C_{d}(\Sigma_{\ell}^{pp})
\]
denote singular $d$-chains of $\Sigma_{\ell}$ in general position with respect to $K$. 
We introduce two string topology operations associated to $K$,
$
\delta_{K}^{Q},\delta_{K}^{N}\colon C_{k}(\Sigma_{\ell})\to C_{k-1}(\Sigma_{\ell+1}).
$
If $\sigma$ is a generic $d$-simplex then $\delta^{Q}_{K}(\sigma)$ is the chain parameterized by the locus in $\sigma$ of strings with components in $S^{3}$ that intersect $K$ at interior points. The operation splits the curve at such intersection and inserts a spike in $L_{K}$, see \cite{CELN}. The operation $\delta_{K}^{Q}$ is defined similarly exchanging the role of $\R^{3}$ and $L_{K}$. There are also similar operations $\delta_{p}^{Q},\delta_{p}^{N}\colon C_{k}(\Sigma_{\ell})\to C_{k-2}(\Sigma_{\ell+1})$ at $p$ that will play less of a role here.

Let $\partial$ denote the singular differential on $C_{\ast}(\Sigma_{\ell})$ and let
$C_{m}=\bigoplus_{k+\ell/2=m} C_{k}(\Sigma_{\ell})$.
We introduce a Pontryagin product which concatenates strings at $p$. We write $R^{\rm st}_{KK}$, $R^{\rm st}_{Kp}$, and $R^{\rm st}_{pK}$ for the degree 0 homology of the corresponding summands of $C_{\ast}$.

\begin{prp}\label{prp:iso1}\cite{CELN,ENS}
The map $d=\partial + \delta_{K}^{Q} + \delta_K^{N} + \delta_{p}^{Q} +\delta_{p}^{N}$ is a differential on $C_{\ast}$. The homology of $d$ in degree 0 is the cokernel of $\partial+ \delta_{K}^{Q} + \delta_K^{N}\colon C_{1}\to C_{0}$ (where $\delta_{p}^{Q}$ and $\delta_{p}^{N}$ vanishes for degree reasons) and is as follows:
\[
 R^{\rm st}_{KK}\approx \hat R + R(1-\mu),\quad
 R^{\rm st}_{Kp}\approx R,\quad
 R^{\rm st}_{pK}\approx R(1-\mu),
\]
where $R=\Z[\pi_{1}(\R^{3}-K)]$ and $\hat R=\Z[\pi_{1}(\Lambda_{K})]$.
\end{prp}

We next consider a geometric chain map of algebras $\Phi\colon CE(\Lambda)\to C_{\ast}$, where the multiplication on $C_{\ast}$ is given by chain level concatenation of broken strings. The map is defined as follows on generators. If $a$ is a Reeb chord let $\mathcal{M}(a;\Sigma)$ denote the moduli space of holomorphic disks in $T^{\ast}S^{3}$ with boundary on $\Sigma$ and Lagrangian intersection punctures at $K$. The evaluation map gives a chain of broken strings for each $u\in\mathcal{M}(a;\Sigma)$. Let $[\mathcal{M}(a;\Sigma)]$ denote the chain of broken strings carried by the moduli space and define
$\Phi(a)=[\mathcal{M}(a;\Sigma)]$.

\begin{prp}\label{prp:iso2}\cite{ENS}
The map $\Phi$ is a chain map. It induces an isomorphism 
\[ 
\left(R_{Kp},R_{pK},R_{KK}\right)\to\left(R_{Kp}^{\rm st},R_{pK}^{\rm st},R_{KK}^{\rm st}\right)
\]
that intertwines the product $m$ and the Pontryagin product at $p$.
\end{prp}

\begin{proof}[Proof of Theorem \ref{t:complete}]
Propositions \ref{prp:iso1} and \ref{prp:iso2} imply that the knot contact homology triple knows the group ring of the knot group and the action of $\lambda$ and $\mu$. Properties of left-orderable groups together with Waldhausen's theorem then give the result, see \cite{ENS} for details. 	
\end{proof}

\subsection{Partially wrapped Floer cohomology and Legendrian surgery}
The knot contact homology of the previous section can also be interpreted, via Legendrian surgery, in terms of partially wrapped Floer cohomology that in turn is connected to the micro-local sheaves used by Shende \cite{shende} to prove the completness result in Theorem \ref{t:complete}. We give a very brief discussion and refer to \cite[Section 6]{ENSarxiv} for more details. 

To a knot $K\subset \R^{3}$ we associate a Liouville sector $W_K$ with Lagrangian skeleton $L=\R^3\cup L_K$, this roughly means that $L$ is a Lagrangian subvariety and that $W_{K}$ is a regular neighborhood of $L$, see \cite{sylvan,GPS}. More precisely, $W_K$ is obtained by attaching the cotangent bundle $T^{\ast} [0,\infty)\times\Lambda_{K}$ to $T^{\ast} \R^{3}$ along $\Lambda_{K}\subset ST^{\ast}\R^{3}$. We let $C_{K}$ denote the cotangent fiber at $q\in [0,\infty)\times \Lambda_{K}$ and $C_{p}$ the cotangent fiber at $p\in \R^{3}$. Such handle attachments were considered in \cite{EL} where it was shown that there exists a natural surgery quasi-isomorphism $\Phi\colon CE(\Lambda_{K})\to CW^{\ast}(C_{K})$, where $CW^{\ast}$ denotes wrapped Floer cohomology. There are directly analogous quasi-isomorphisms
\[ 
CE_{01}^{(1)}\to CW^{\ast}(C_{K},C_{p}),\quad
CE_{10}^{(1)}\to CW^{\ast}(C_{p},C_{K}),\quad
CE_{11}^{(0)}\to CW^{\ast}(C_{p},C_{p}),
\]
under which the product $m$ corresponds to the usual triangle product $m_{2}$ on $CW^{\ast}$.

In \cite{shende}, the conormal torus $\Lambda_K$ of a knot $K\subset \R^{3}$ was studied via the category of sheaves microsupported in $L$. This sheaf category can also be described as the category of modules over the wrapped Fukaya category of $W_{K}$ which is generated by the two cotangent fibers $C_{K}$ and $C_{p}$. The knot contact homology triple with $m$ then have a natural interpretation as calculating morphisms in a category equivalent to that studied in \cite{shende}.

\section{Augmentations and the Gromov-Witten disk potential}\label{Sec:aug+disk}
Let $K\subset S^{3}$ be a knot and let $L_{K}$ denote its conormal Lagrangian. Shifting $L_{K}$ along the 1-form dual to its unit tangent vector we get a non-exact Lagrangian that is disjoint from the 0-section. We identify the complement of the $0$-section in $T^{\ast}S^{3}$ with the complement of the 0-section in the resolved conifold $X$. Under this identification, $L_{K}$ becomes a uniformly tame Lagrangian, see \cite{koshkin}, which is asymptotic to $\R\times\Lambda_{K}\subset\R\times S T^{\ast}S^{3}$ at infinity. The first condition implies that $L_{K}$ can be used as boundary condition for holomorphic curves and the second that at infinity, holomorphic curves on $(X,L_{K})$ can be identified with the $\R$-invariant holmorphic curves of $(\R\times ST^{\ast}S^{3},\R\times\Lambda_{K})$.

Since $c_{1}(X)=0$ and the Maslov class of $L_{K}$ vanishes, the formal dimension of any holomorphic curve in $X$ with boundary on $L_{K}$ equals $0$. Fixing a perturbation scheme one then gets a 0-dimensional moduli space of curves. Naively, the open Gromov-Witten invariant of $L_{K}$ would be the count of these rigid curves. Simple examples however show that such a count is not invariant under deformations, contradicting what topological string theory predicts. 

To resolve this problem on the Gromov-Witten side, we count more involved configurations of curves that we call \emph{generalized curves}. In this section we consider the simpler case of disks and then in Section \ref{Sec:SFT} the case of general holomorphic curves. The problems of open Gromov-Witten theory in this setting was studied from the mathematical perspective also by Iacovino \cite{iacovino1,iacovino2}. From the physical perspective, the appearance of more complicated configurations then bare holomorphic curves seems related to boundary terms in the path integral localized on the moduli space of holomorphic curves with boundary which, unlike in the case of closed curves, has essential codimension one boundary strata.   

\subsection{Augmentations of non-exact Lagrangians and disk potentials}
We will construct augmentations induced by the non-exact Lagrangian filling $L_{K}\subset X$. In order to explain how this works we first consider the case of the exact filling $L_{K}\subset T^{\ast} S^{3}$. The exact case is a standard ingredient in the study of Chekanov-Eliashberg dg-algebras, see e.g.~\cite{EHK}. Consider the algebra  $\mathcal{A}_{K}$ with coefficients in $\C[e^{\pm x},e^{\pm p},Q^{\pm 1}]$. Here we set $e^{p}=1$ since $p$ bounds in $L_{K}$ and $Q=1$ since the cotangent fiber sphere bounds in $S T^{\ast}S^{3}$. If $a$ is a Reeb chord of $\Lambda_{K}$, we let $\mathcal{M}_{n}(a)$ denote the moduli space of holomorphic disks with positive puncture at $a$ and boundary on $L_{K}$ that lies in the homology class $nx$. Then $\dim(\mathcal{M}_{n}(a))=|a|$ and we define the map $\epsilon_0\colon \mathcal{A}_{K}\to \C[e^{\pm x}]$ on degree $0$ Reeb chords $a$ as
\[ 
\epsilon_{0}(a)=\sum_{n} |\mathcal{M}_{n}(a)|e^{nx}.
\]  
\begin{lma}
The map $\epsilon_{0}\colon \mathcal{A}_{K}|_{Q=1,e^{p}=1}\to\C[e^{\pm x}]$ is a chain map, $\epsilon_{0}\circ d=0$.
\end{lma}
\begin{proof}
Configurations contributing to $\epsilon_{0}\circ d$ are two level broken disks that are in one to one correspondence with the boundary of the oriented 1-manifolds $\mathcal{M}_{n}(c)$, $|c|=1$.	
\end{proof}

We next consider the case of the non-exact Lagrangian filling $L_{K}\subset X$. In this case, $Q=e^{t}$, where $t=[\C P^{1}]\in H_{2}(X)$ and we look for a chain map $\mathcal{A}_{K}\to\C[e^{\pm x},Q^{\pm 1}]$. If $a$ is a Reeb chord, then let $\mathcal{M}_{r,n}(a)$ denote the moduli space of holomorphic disks in $X$ with boundary on $L_{K}$ in relative homology class $rt+nx$. 

Consider first the naive generalization of the exact case and define 
\[
\epsilon'(a)=\sum_{r,n}|\mathcal{M}_{r,n}(a)|Q^{r}e^{nx}.
\]
We look at the boundary of 1-dimensional moduli spaces $\mathcal{M}_{r,n}(c)$, $|c|=1$. 
Unlike in the exact case, two level broken curves do not account for the whole boundary of $\mathcal{M}_{r,n}(c)$ and consequently the chain map equation does not hold. The reason is that there are non-constant holomorphic disks without positive punctures on $L_{K}$ and a 1-dimensional family of disks can split off non-trivial such disks under so called boundary bubbling. Together with two level disks, disks with boundary bubbles account for the whole boundary of the moduli spacce.

The problem of boundary bubbling is well-known in Floer cohomology and was dealt with there using the method of bounding cochains introduced by Fukaya, Oh, Ohta, Ono \cite{FO3}. We implement this method in the current set up by introducing non-compact bounding chains (with boundary at infinity) as follows. 
We use a perturbation scheme to make rigid disks transversely cut out energy level by energy level. For each transverse disk $u$ we also fix a bounding chain $\sigma_{u}$, i.e., $\sigma_{u}$ is a non-compact 2-chain in $L_{K}$ that interpolates between the boundary $\partial u$ and a multiple of a fixed curve $\xi$ in $\Lambda_{K}$ in the longitude homology class $x\in H_{1}(\Lambda_{K})$ at infinity. This allows us to define the Gromov-Witten disk potential as a sum over finite trees $\Gamma$, where there is a rigid disk $u_{v}$ at each vertex $v\in\Gamma$ and for every edge connecting vertices $v$ and $v'$ there is an intersection point between $\partial u_{v}$ and $\sigma_{v'}$ weighted by $\pm\frac12$, according to the intersection number. We call such a tree a \emph{generalized disk} and define the Gromov-Witten disk potential $W_{K}(x,Q)$ as the generating function of generalized disks.

\begin{figure}[htp]
	\centering
	\includegraphics[bb=-150 250 1100 650, width=.8\linewidth]{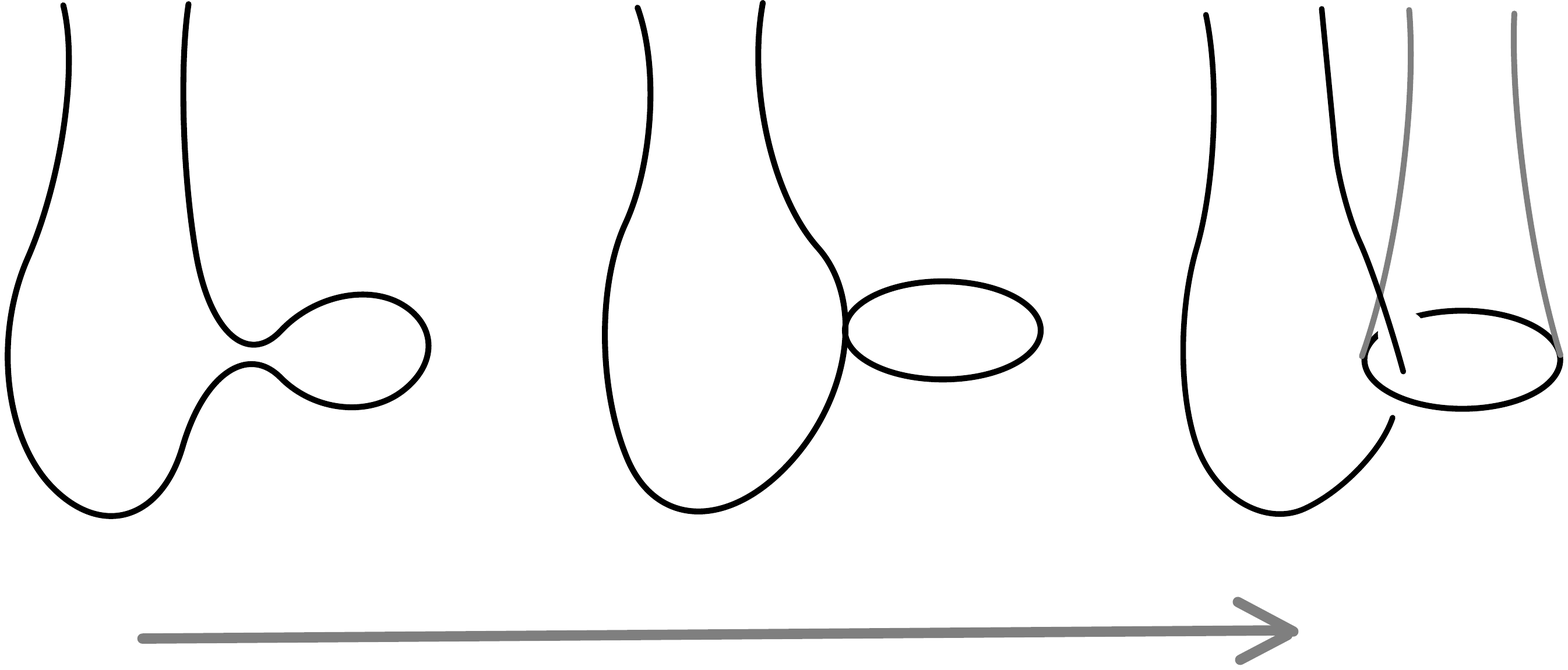}
	\caption{Bounding chains turn boundary breaking into interior points in moduli spaces: the disk family continues as a family of disks with the bounding chain inserted.}
	\label{fig:boundingchain}
\end{figure}

We then define $\mathcal{M}'_{r,n}(a)$ as the moduli space of holomorphic disks with positive puncture at $a$ and with insertion of bounding chains of generalized disks along its boundary such that the total homology class of the union of all disks in the configuration lies in the class $rt+nx$. Let $\epsilon\colon\mathcal{A}_{K}\to\C[e^{\pm x},Q^{\pm 1}]$ be the map
\[ 
\epsilon(a)=\sum_{r,n}|\mathcal{M}_{r,n}'(a)|Q^{r}e^{nx}.
\]

\begin{prp}
If
\begin{equation}\label{eq:augparam} 
p=\frac{\partial W_{K}}{\partial x}
\end{equation}
then $\epsilon$ is a chain map, $\epsilon\circ d=0$. Consequently, \eqref{eq:augparam}  
parameterizes a branch of the augmentation variety and Theorem \ref{t:diskpotential} follows.
\end{prp}  

\begin{proof}
The bounding chains are used to remove boundary bubbling from the boundary of the moduli space, see Figure \ref{fig:boundingchain}. As a consequence the boundary of the moduli space of disks with one positive puncture and with insertions of generalized disks correspond to two level disks with insertions. It remains to count the disks at infinity with insertions. At infinity all bounding chains are multiples of the longitude generator $x$. A bounding chain going $n$-times around $x$ can be inserted $nm$ times in a curve that goes $m$ times around $p$. It follows that the substitution $e^{p}=e^{\frac{\partial W_{K}}{\partial x}}$
corresponds to counting disks with insertions.
\end{proof}

\begin{cor}
The Gromov-Witten disk potential $W_{K}$ is an analytic function.
\end{cor}

\begin{proof}
The defining equation of the augmentation variety can be found from the knot contact homology differential by elimination theory. It is therefore an algebraic variety and the Gromov-Witten disk potential in \eqref{eq:augparam} is an analytic function.
\end{proof}

\subsection{Augmentation varieties in basic examples}
We calculate the augmentation variety from the formulas in Section \ref{sec:ex1}. 
\subsubsection{The unknot}
The augmentation polynomial for the unknot $U$ is determined directly by \eqref{eq:unknotdiff}: the algebra admits an augmentation exactly when $dc=0$ and
\[ 
\Aug_{U}=1-e^{x}-e^{p}+Qe^{x}e^{p}.
\] 
\subsubsection{The trefoil}
We need to find the locus where the right hand sides in Section \ref{ssec:trefoildiff} has common roots. The augmentation polynomial is found as:
\begin{align*}
\Aug_{T} &=(e^{x}e^{2p} + Q^{2})a_{12}(e^{p}(dc_{21}) - Q(dc_{22})) - (e^{x}e^{2p} + Q^2)(Q(dc_{21}) + e^{x}e^{p}d(c_{22}))\\ 
&+ e^{x}(e^{2p} - Q)(e^{p}(dc_{21}) - Q(dc_{22})) + e^{x}(e^{2p} - Q)(e^{x}e^{2p} + Q^{2})(db_{12})\\
&= e^{2x}(e^{4p}-e^{3p}) + e^{x}(e^{4p}-e^{3p}Q+2e^{2p}(Q^{2}-Q)-e^{p}Q^{2}+Q^{2}) -(e^{p}Q^{3}-Q^{4}).
\end{align*}

\section{Legendrian SFT and open Gromov-Witten theory}\label{Sec:SFT}
This section concerns the higher genus counterpart of the results in Section \ref{Sec:aug+disk}.

\subsection{Additional geometric data for Legendrian SFT of knot conormals}
We outline a definition of relevant parts of Legendrian SFT (including the open Gromov-Witten potential), for the Lagrangian conormal $L_{K}$ of a knot $K\subset S^{3}$ in the resolved conifold, $L_{K}\subset X$. As in the case of holomorphic disks, see Section \ref{Sec:aug+disk}, the main point of the construction is to overcome boundary bubbling. In the disk case there is a 1-dimensional disk that interacts through boundary splitting/crossing with rigid disks. Since the moving disk is distinguished from the rigid disks, it is sufficient to use bounding chains for the rigid disks only.

In the case of higher genus curves there is no such separation. A 1-dimensional curve can boundary split on its own. To deal with this, we introduce additional geometric data that defines what might be thought of as dynamical bounding chains.  We give a brief description here and refer to \cite{EN} for more details. The construction was inspired by self linking of real algebraic links (Viro's encomplexed writhe, \cite{viro}) as described in \cite{shade}.

\subsubsection{An auxiliary Morse function}
Consider a Morse function $f\colon L_K\to\R$ without maximum and with the following properties.
The critical points of $f$ lie on $K$ and are: a minimum $\kappa_{0}$ and an index 1 critical point $\kappa_{1}$. Flow lines of $\nabla f$ connecting $\kappa_{0}$ to $\kappa_{1}$ lie in $K$ and outside a small neighborhood of $K$, $\nabla f$ is the radial vector field along the fiber disks in $L_{K}\approx K\times\R^{2}$. Note that the unstable manifold $W^{\rm u}(\kappa_{1})$ of $\kappa_{1}$ is a disk that intersects $\Lambda_{K}$ in the meridian cycle $p$.

\subsubsection{A 4-chain with boundary twice $L_{K}$}
Start with a $3$-chain $\Gamma_{K}\subset ST^{\ast}S^{3}$ with the following properties: $\partial \Gamma_{K}= 2\cdot\Lambda_{K}$, near the boundary $\Gamma_{K}$ agrees with the union of small length $\epsilon>0$ flow lines of $\pm R$ starting on $\Lambda_{K}$, and $\Gamma_{K}-\partial\Gamma_{K}$ is disjoint from $\Lambda_{K}$, see \cite{EN}. Identify $\left([0,\infty)\times ST^{\ast}S^{3},[0,\infty)\times\Lambda_{K}\right)$ with $(X,L_{K})-(\bar X,\bar L_{K})$, where $(\bar X,\bar L_{K})$ is compact, and let $C_{K}^{\infty}=[0,\infty)\times\Gamma_{K}$.

Consider the vector field $v(q)=\frac{\nabla f(q)}{|\nabla f(q)|}$, $q\in L_{K}-\{ \kappa_{0},\kappa_{1}\}$ and let $G$ be the closure of the length $\epsilon>0$ half rays of $\pm J v(q)$ starting at $q\in L_{K}$ in $\bar L_{K}$ and $G'$ its boundary component that does not intersect $L_{K}$. A straightforward homology calculation shows that there exists a  $4$-chain $C_{K}^{0}$ in $X-L_{K}$ with boundary $\partial C_{K}^{0}=G' \cup \partial C_{K}^{\infty}$.
Define $C_{K}= C_{K}^{\infty}\times[0,\infty) \cup C_{K}^{0}\cup G$.
Then $C_{K}$ is a $4$-chain with regular boundary along $2\cdot L_{K}$ and inward normal $\pm J\nabla f$. Furthermore, $C_{K}$ intersects $L_{K}$ only along its boundary and is otherwise disjoint from it.
We remark that in order to achieve necessary transversality, we also need to perturb the Morse function and the chain slightly near the Reeb chord endpoints in order to avoid intersections with trivial strips, see \cite{EN}.

\subsection{Bounding chains for holomorphic curves}\label{sec:boundingchains}
We next associate a bounding chain to each holomorphic curve $u\colon (\Sigma,\partial\Sigma)\to (X,L_{K})$ in general position with respect to $\nabla f$ and $C_{K}$. Consider first the case without punctures. The boundary $u(\partial\Sigma)$ is a collection of closed curves contained in a compact subset of $L_{K}$. By general position, $u(\partial\Sigma)$ does not intersect the stable manifold of $\kappa_{1}$. Define $\sigma'_u$ as the union of all flow lines of $\nabla f$ that starts on $u(\partial\Sigma)$. Since $f$ has no index $2$ critical points and since $\nabla f$ is vertical outside a compact, $\sigma'_u\cap(\{T\}\times\Lambda_{K})$ is a closed curve, independent of $T$ for all sufficiently large $T>0$. Let $\partial_{\infty}\sigma'_u\subset\Lambda_{K}$ denote this curve and assume that its homology class is $nx + mp\in H_{1}(\Lambda_{K})$. Define the bounding chain $\sigma_{u}$ of $u$ as
\begin{equation}\label{eq:defboundingchain1}
\sigma_{u}=\sigma_{u}'-m\cdot W^{\mathrm{u}}(\kappa_{1}).
\end{equation}
Then $\sigma_{u}$ has boundary $\partial\sigma_{u}=\partial u$ and boundary at infinity $\partial^{\infty}\sigma_{u}$ in the class $nx+0p$. 

Consider next the general case when $u\colon(\Sigma,\partial\Sigma)\to(X,L_{K})$ has punctures at Reeb chords $c_{1},\dots, c_{m}$. Let $\delta_{j}$ denote the capping disk of $c_{j}$ and let $\bar X_{T}=\bar X\cup ([0,T]\times ST^{\ast}S^{3})$. Fix a sufficiently large $T>0$ and replace $u(\partial\Sigma)$ in the construction of $\sigma_u'$ above by the boundary of the chain 
$(u(\Sigma)\cap \bar X_{T}) \cup \bigcup_{j=1}^{m}\delta_{j}$
and then proceed as there. This means that we cap off the holomorphic curve by adding capping disks and construct a bounding chains of this capped disk.

\subsection{Generalized holomorphic curves and the SFT-potential}\label{sec:gencurves}
The SFT counterpart of the chain map equation for augmentations is derived from 1-dimensional moduli spaces of generalized holomorphic curves. The moduli spaces are stratified and the key point of our construction is to patch the 1-dimensional strata in such a way that all boundary phenomena in the compact part of $(X,L_{K})$ cancel out, leaving only splitting at Reeb chords and intersections with bounding chains at infinity. We start by describing the curves in the 1-dimensional strata. 

As in the disk case we assume we have a perturbation scheme for transversality. Again the perturbation is inductively constructed, we first perturb near the simplest curves (lowest energy and highest Euler characteristic) and then continue inductively in the hierarchy of curves, making all holomorphic curves transversely cut out and transverse with respect to the Morse data fixed. We also need transversality with respect to $C_{K}$ that we explain next. A holomorphic curve $u$ in general position has tangent vector along the boundary everywhere linearly independent of $\nabla f$. Let the \emph{shifting vector field} $\nu$ along $\partial u$ be a vector field that together with the tangent vector of $\partial u$ and $\nabla f$ gives a positively oriented triple. Let $\partial u_{\nu}$ denote $\partial u$ shifted slightly along $\nu$. By construction $\partial u_{\nu}$ is disjoint from a neighborhood of the boundary of $\sigma_{u}$. Let $u_{J\nu}$ denote $u$ shifted slightly along an extension of $J\nu$ supported near the boundary of $u$. We chose the perturbation so that $u_{J\nu}$ is transverse to $C_{K}$. 

With such perturbation scheme constructed we define generalized holomorphic curves to consist of the following data.
\begin{itemize}
	\item A finite oriented graph $\Gamma$ with vertex set $V_\Gamma$ and edge set $E_\Gamma$. 
	\item To each $v\in V_\Gamma$ is associated a (generic) holomorphic curve $u^v$ with boundary on $L_{K}$ (and possibly with positive punctures). 
	\item To each edge $e\in E_\Gamma$ that has its endpoints at distinct vertices, $\partial e=v_+-v_-$, $v_+\ne v_-$, is associated an intersection point of the boundary curve $\partial u^{v_{-}}$ and the bounding chain $\sigma_{u^{v_+}}$. 
	\item
	To each edge $e\in E_\Gamma$ which has its endpoints at the same vertex $v_0$, $\partial e=v_0-v_0=0$, is associated either an intersection point in $\partial u^{v_{0}}_{\nu}\cap \sigma_{u^{v_0}}$  or an intersection point in $u^{v_0}_{J\nu}\cap C_{K}$.  
\end{itemize}
We call such a configuration a \emph{generalized holomorphic curve over $\Gamma$} and denote it $\Gamma_{\mathbf{u}}$, where $\mathbf{u}=\{u_{v}\}_{v\in V_\Gamma}$ lists the curves at the vertices.

\begin{rmk}
Several edges of a generalized holomorphic curve may have the same intersection point associated to them.  
\end{rmk}

We define the Euler characteristic of a generalized holomorphic curve $\Gamma_{\mathbf{u}}$ as
\[
\chi(\Gamma_{\mathbf{u}})= \sum_{v\in V_\Gamma} \chi(u_v) -\# E_\Gamma,
\]
where $\# E_\Gamma$ denotes the number of edges of $\Gamma$, and the dimension of the moduli space containing $\Gamma_{\mathbf{u}}$ as 
\[
\dim(\Gamma_{\mathbf{u}})=\sum_{v\in V_\Gamma} \dim(u_v),
\]
where $\dim(u_{v})$ is the formal dimension of $u_{v}$.

In particular, if $\dim(\Gamma_{\mathbf{u}})=0$ then $u_{v}$ is rigid for all $v\in V_\Gamma$ and if $\dim(\Gamma_{\mathbf{u}})=1$ then $\dim(u_{v})=1$ for exactly one $v\in V_\Gamma$ and $u_v$ is rigid for all other $v\in V_\Gamma$. The relative homology class represented by $\Gamma_{\mathbf{u}}$ is the sum of the homology classes of the curves $u_{v}$ at its vertices, $v\in V_\Gamma$. 

We define the SFT-potential to be the generating function of generalized rigid curves over graphs $\Gamma$ as just described:
\[
\mathbf{F}_{K} = \sum_{m,k,\mathbf{c^{+}}} F_{m,k,\chi,\mathbf{c}^{+}}\, g_{s}^{-\chi+\ell(\mathbf{c}^{+})}\, e^{mx} Q^{k}\, \mathbf{c}^{+},
\]
where $F_{m,k,\chi,\mathbf{c}^{+}}$ counts the algebraic number of generalized curves $\Gamma_{\mathbf{u}}$ in homology class $mx+kt\in H_{2}(X,L_{K})$ with $\chi(\Gamma_{\mathbf{u}})=\chi$ and with positive punctures according to the Reeb chord word $\mathbf{c}^{+}$. A generalized curve $\Gamma_{\mathbf{u}}$ contributes to this sum by the product of the weights of the curves at its vertices (the count coming from the perturbation scheme) times $\pm \frac12$ for each edge where the sign is determined by the intersection number.

\begin{rmk}
For computational purposes we note that we can rewrite the sum for $\mathbf{F}_{K}$ in a simpler way. Instead of the complicated oriented graphs with many edges considered above, we look at unoriented graphs with at most one edge connecting every pair of distinct vertices and no edge connecting a vertex to itself. We call such graphs \emph{simple graphs}. We map complicated graphs to simple graphs by collapsing edges to the basic edge and removing self-edges. Then the contribution from all graphs lying over a simple graph is given the product of weights at the vertices times the product of $e^{\lk_{e} g_{s}}$, where the linking coefficient of an edge $e$ connecting vertices corresponding to the curves $u$ and $u'$ is the intersection number $\sigma_{u}\cdot \partial u'=\partial u\cdot \sigma_{u'}$, and $e^{\frac12\slk_{v} g_{s}}$, where the linking coefficient $\slk_{v}$ of a vertex $v$ is the sum of intersection numbers $\partial u_{\nu}\cdot \sigma_{u}+ u_{J\nu}\cdot C_{K}$, where $u$ is the curve at $v$.
\end{rmk}

\subsection{Compactification of 1-dimensional moduli spaces}
The generalized holomorphic curves that we defined in Section \ref{sec:gencurves} constitute the  open strata of the 1-dimensional moduli. More precisely, the generalized curve $\Gamma_{\mathbf{u}}$ has a \emph{generic} curve of dimension one at exactly one vertex. Except for the usual holomorphic degenerations in 1-parameter families, there are new boundary phenomena arising from the 1-dimensional curve becoming non-generic relative $\nabla f$ and $C_{K}$. More precisely we have the following description of the boundary of 1-dimensional starta of generalized holomorphic curves (we write $u_{v}$ for the curve at vertex $v\in \Gamma_{\mathbf{u}}$).

\begin{lma}\cite{EN}\label{l:degenrations}
Generic degenerations of the holomorphic curves $u_{v}$ at the vertices $v\in V_{\Gamma}$ are as follows (see Figure \ref{fig:degenerations}):  
\begin{itemize}
	\item[$(1)$] Splitting at Reeb chords.   
	\item[$(2)$] Hyperbolic boundary splitting.
	\item[$(3)$] Elliptic boundary splitting.  
\end{itemize}
Generic degenerations with respect to $\nabla f$, $C_{K}$, and capping paths are as follows: 
\begin{itemize}
	\item[$(4)$] Crossing the stable manifold of $\kappa_{1}$: the boundary of the curve intersects the stable manifold of $\kappa_{1}$.
	\item[$(5)$] Boundary crossing: a point in the boundary mapping to a bounding chain moves out across the boundary of a bounding chain.
	\item[$(6)$] Interior crossing: An interior marked point mapping to $C_K$ moves across the boundary $L_{K}$ of $C_{K}$.
	\item[$(7)$] Boundary kink: The boundary of a curve becomes tangent to $\nabla f$ at one point.
	\item[$(8)$] Interior kink: A marked point mapping to $C_K$ moves to the boundary in the holomorphic curve.
	\item[$(9)$] The leading Fourier coefficient at a positive puncture vanishes.
\end{itemize}
\end{lma}
 
\begin{figure}[htp]
	\centering
	\includegraphics[width=.75\linewidth]{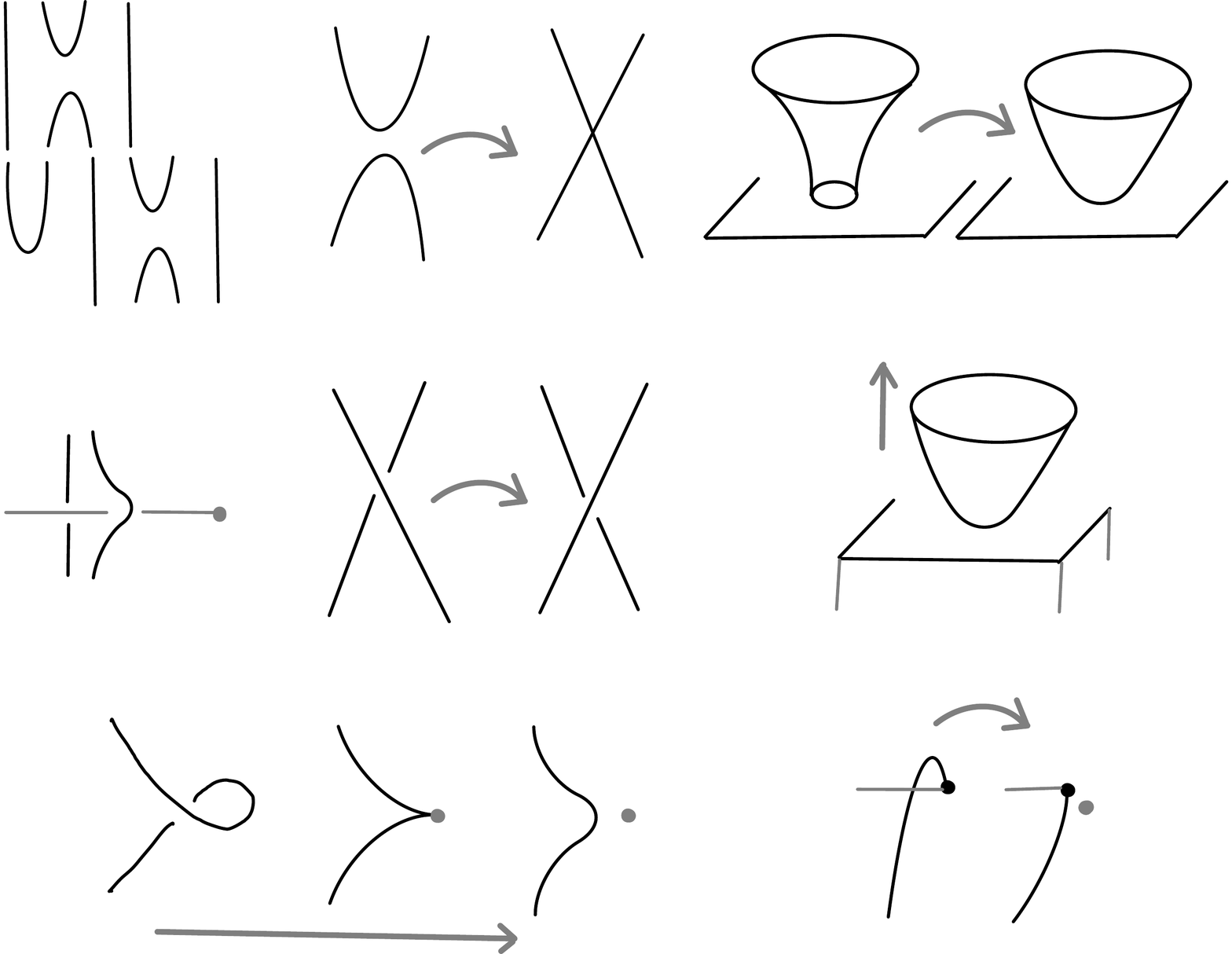}
	\caption{Degenerations in Lemma \ref{l:degenrations}. Top row: $(1),(2),(3)$, middle $(4)$ (the dot is $\kappa_{1}$), $(5),(6)$, bottom $(7),(8)$ together, and $(9)$ (gray dot represents $u_{J\nu}\cap C_{K}$).}
	\label{fig:degenerations}
\end{figure}

\begin{prp}\label{prp:cancel}
Boundaries of 1-dimensional strata of generalized holomorphic curves cancel out according to the following.	
\begin{itemize}
\item[$(i)$] The moduli space of generalized holomorphic curves does not change under degenerations $(4)$ and $(9)$.
\item[$(ii)$] Boundary splitting $(2)$ cancel with boundary crossing $(5)$.
\item[$(iii)$] Elliptic splitting $(3)$ cancel with interior crossing $(6)$.
\item[$(iv)$] Boundary kinks $(7)$ cancel interior kinks $(8)$.
\end{itemize}
\end{prp}

\begin{proof}
Consider $(i)$. For $(4)$, observe that as the boundary crosses the stable manifold of $\kappa_{1}$, the change in flow image is compensated by the change in the number of unstable manifold added. The invariance under $(9)$ follows from a straightforward calculation using Fourier expansion near the Reeb chord: an intersection with the capping disk boundary turns into an intersection with $C_{K}$.

Consider $(iv)$. A calculation in a local model for a generic tangency with $\nabla f$ shows that the self intersection of the boundary turns into an intersection with $C_{K}$. (This uses that the normal vector field of $C_{K}$ is $\pm J\nabla f$.)

Consider $(iii)$. Unlike $(i)$ and $(iv)$ this involves gluing holomorphic curves and therefore, as we  will see, the details of the perturbation scheme (which also has further applications, see \cite{ES}).  

At the hyperbolic boundary splitting we find a holomorphic curve with a double point that can be resolved in two ways, $u_+$ and $u_-$. Consider the two moduli spaces corresponding to $m$ insertions at the corresponding intersection points between $\partial u_+$ and $\sigma_{u_-}$ and $\partial u_-$ and $\sigma_{u_{+}}$. 

To obtain transversality at this singular curve for curves of any Euler characteristic we must separate the intersection points with the bounding chain. To this end, we use a perturbation scheme with multiple bounding chains that time-orders the boundary crossings. Each, now distinct, crossing can then be treated as a usual gluing. Consider gluing at $m$ intersection points as $\partial u_{-}$ crosses $\sigma_{u_+}$. This gives a curve of Euler characteristic decreased by $m$ and orientation sign $\epsilon^{m}$, $\epsilon=\pm 1$. Furthermore, at the gluing, the ordering permutation acts on the gluing strips and each intersection point is weighted by $\frac12$. (The reason for the factor $\frac12$ is that we count intersections between boundaries and bounding chains twice, for distinct curves both $\partial u\cap \sigma_v$ and $\partial v\cap \sigma_u$ contribute.) This gives a moduli space of additional weight
\[
\epsilon^{m}\frac{1}{2^{m}m!} g_{s}^{m}.
\]  
The only difference between these configurations and those associated with the opposite crossing is the orientation sign. Hence the other gluing when $\partial u_{+}$ crosses $\sigma_{u_{-}}$ gives the weight
\[
(-1)^{m}\epsilon^{m}\frac{1}{2^{m}m!} g_{s}^{m}.
\]

Noting that the original moduli space is oriented towards the crossing for one configuration and away from it for the other we find that the two gluings cancel if $m$ is even and give a new curve of Euler characteristic decreased by $m$ and of weight $\frac{2}{2^{m}m!}$ if $m$ is odd. Counting ends of moduli spaces we find that the curves resulting from gluing at the crossing count with a factor
\begin{equation}\label{eq:magicfactor}
e^{\frac12 g_{s}}-e^{-\frac12 g_{s}},
\end{equation}
which cancels the change in linking number.

Cancellation $(iii)$ follows from a gluing argument analogous to $(ii)$: 
The curve with an interior point mapping to $L_K$ can be resolved in two ways, one curve $u^+$ that intersects $C_{K}$ at a point in the direction $+J\nabla f$ and one $u^-$ that intersects $C_{K}$ at a point in the direction $-J\nabla f$. A constant disk at the intersection point can be glued to the family of curves at the intersection with $L_{K}$. As in the hyperbolic case we separate the intersections and time order them to get transversality at any Euler characteristic. We then apply usual gluing and note that the intersection sign is part of the orientation data for the gluing problem, the calculation of weights is exactly as in the hyperbolic case above. (This time the $\frac12$-factors comes from the boundary of $C_K$ being twice $L_K$, $\partial C_{K}=2[L_K]$.) We find again that glued configurations corresponds to multiplication by
\[ 
e^{\frac12 g_{s}}-e^{-\frac12 g_{s}},
\]
and cancels the difference in counts between $u^{+}_{J\nu}\cdot C_{K}$ and $u^{-}_{J\nu}\cdot C_{K}$. 
\end{proof}

\subsection{The SFT equation}
We let $\mathbf{H}_{K}$ denote the count of generalized holomorphic curves $\Gamma_{\mathbf{u}}$, in $\R\times S T^{\ast}S^{3}$, rigid up to $\R$-translation. Such a generalized curve lies over a graph that has a main vertex corresponding to a curve of dimension 1, at all other vertices there are trivial Reeb chord strips. Consider such a generalized holomorphic curve $\Gamma_{\mathbf{u}}$. We write $\mathbf{c}^{+}(\mathbf{u})$ and $\mathbf{c}^{-}(\mathbf{u})$ for the monomials of positive and negative punctures of $\Gamma_{\mathbf{u}}$, write $w(\mathbf{u})$ for the weight of $\Gamma_{\mathbf{u}}$, $m(\mathbf{u})x+n(\mathbf{u})p+l(\mathbf{u})t$ for its homology class, and $\chi(\mathbf{u})$ for the Euler characteristic of the generalized curve of $\Gamma_{\mathbf{u}}$. Define the SFT-Hamiltonian
\[
\mathbf{H}_{K}=\sum_{\dim(\Gamma_{\mathbf{u}})=1} w(\mathbf{u})\;  g_{s}^{-\chi(\mathbf{u})+\ell(\mathbf{c}^{+}(\mathbf{u}))}\; e^{m(\mathbf{u})x+n(\mathbf{u})p+l(\mathbf{u})t}\;\mathbf{c}^{+}(\mathbf{u})\,\partial_{\mathbf{c}^{-}(\mathbf{u})},  
\]
where the sum ranges over all generalized holomorphic curves. As above this formula can be simplified to a sum over simpler graphs with more elaborate weights on edges. 

\begin{lma}\label{l:quantization}
	Consider a curve $u$ at infinity in class $mx+np+kt$. The count of the corresponding generalized curves with insertion along $\partial u$ equals
	\[ 
	e^{-\mathbf{F}_{K}} e^{mx}Q^{k} e^{ng_{s}\frac{\partial}{\partial x}}e^{\mathbf{F}_{K}}.
	\]
\end{lma}
\begin{proof}
	Contributions from bounding chains of curves inserted $r$ times along $np$ corresponds to multiplication by
	\[ 
	n^{r}\frac{1}{r!}g_{s}^{-r}\sum_{r_{1}+\dots+r_{j}=r}
	\frac{\partial^{r_{1}} \mathbf{F}_{K}}{\partial x^{r_{1}}}
	\dots
	\frac{\partial^{r_{j}} \mathbf{F}_{K}}{\partial x^{r_{j}}},
	\]	
	where a factor
	$\frac{\partial^{s} \mathbf{F}_{K}}{\partial x^{s}}$ corresponds to attaching the bounding chain of a curve $s$ times.
\end{proof}

\begin{thm}\label{t:sftmastereq}
	If $K$ is a knot and $L_{K}\subset X$ its conormal Lagrangian then the SFT equation
	\begin{equation}\label{eq:sft}
	e^{-\mathbf{F}_{K}} \ \mathbf{H}_{K}|_{p=g_s\frac{\partial}{\partial x}} \ e^{\mathbf{F}_{K}} =0
	\end{equation}
	holds.
\end{thm}

\begin{proof}
	Lemma \ref{l:degenrations}, Proposition \ref{prp:cancel}, and Lemma \ref{l:quantization} show that the terms in left hand side of \eqref{eq:sft} counts the ends of a compact oriented 1-dimensional moduli space.
\end{proof}    

\begin{rmk}
	We point out that Lemma \ref{l:quantization} gives an enumerative geometrical meaning to the standard quantization scheme $p=g_{s} \frac{\partial}{\partial x}$ by counting insertions of bounding chains. See \cite[Section 3.3]{Ekholmoverview} for a related path integral argument.
\end{rmk}

\subsection{Framing and Gromov-Witten invariants}
Lemma \ref{l:degenrations} and Proposition \ref{prp:cancel} imply that the open Gromov-Witten potential of $L_{K}$ is invariant under deformation. Recall from Section \ref{sec:physics} that dualities between string and gauge theories imply that
\[ 
\Psi_{K}(x,Q)=e^{F_{K}(x,Q)}=\sum_{m} H_{m}(e^{g_{s}},Q)e^{mx},
\]
where $H_{m}$ is the $m$-colored HOMFLY-PT polynomial. It is well-known that the colored HOMFLY-PT polynomial depends on framing. We derive this dependence here using our definition of generalized holomorphic curves. Assume that $\Psi_{K}$ above is defined for a framing $(x,p)$ of $\Lambda_{K}$. Then other framings are given by $(x',p')=(x+rp,p)$ where $r$ is an integer. Let $\Psi_{K}^{r}(x',Q)$ denote the wave function defined using the framing $(x',p')$.  
\begin{thm}\label{t:framing}
	If $\Psi_{K}(x,Q)$ is as above then
	\[ 
	\Psi_{K}^{r}(x',Q)=\sum_{m} H_{m}(e^{g_{s}},Q)\,  e^{m^{2}rg_{s}} e^{mx'}.
	\]
\end{thm}

\begin{proof}
	Note first that the actual holomorphic curves are independent of the framing. The change thus comes from the bounding chains: the boundaries at infinity $\partial_{\infty}\sigma_{u}$ must be corrected to lie in multiples of the new preferred class $x'$. Thus, for a curve that goes $m$ times around the generator of $H_{1}(L_{K})$, we must correct the bounding chain adapted to $x$ by adding $mr W^{\rm u}(\kappa_{1})$. Under such a change, the linking number in $L_{K}$ in this class changes by $m^{2}r$.
\end{proof}

\subsection{Quantization of the augmentation variety in basic examples}
\subsubsection{The unknot}
Using Morse flow trees it is easy to see that there are no higher genus curves with boundary on $\Lambda_{U}$. As with the augmentation polynomial, there are no additional operators to eliminate for the unknot and $\mathbf{H}_{U}$ gives the operator equation directly:
\[ 
\qAug_{U}=1-e^{\hat x}-e^{\hat p}-Qe^{\hat x}e^{\hat p},
\]
which agrees with the the recursion relation for the colored HOMFLY-PT, see e.g.~\cite{AV}.

\subsubsection{The trefoil}
It can be shown \cite{EN} that there are no higher genus curves with boundary on $\Lambda_{T}$. The SFT Hamiltonian can again be computed from disks with flow lines attached. If $c$ is a chord with $|c|=1$, we write $H(c)$ for the part of the Hamiltonian $\mathbf{H}_{T}$ with a positive puncture at $c$ and leave out $c$ from the notation. Then relevant parts of the Hamiltonian are:
\begin{align*}
H(b_{12}) &= e^{-\hat x} \partial_{a_{12}} - \partial_{a_{21}} +\mathcal{O}(a)\\
H(c_{11}) &= 
e^{\hat x}e^{\hat p} - e^{-g_{s}}e^{\hat x} -((1+e^{-g_{s}})Q-e^{\hat p}) \partial_{a_{12}} - Q \partial^2_{a_{12}} \partial_{a_{21}} + \mathcal{O}(a) \\
H(c_{21}) &= Q - e^{\hat p} + e^{\hat x}e^{\hat p} \partial_{a_{21}} 
+ Q \partial_{a_{12}}\partial_{a_{21}} +  (e^{-g_{s}}-1) e^{\hat x} a_{12}  \\
& \qquad+ (e^{-g_{s}}-1) Q a_{12} \partial_{a_{12}}
+ \mathcal{O}(a^2) \\
H(c_{22}) &= e^{\hat p}-1-Q \partial_{a_{21}}+e^{\hat p} \partial_{a_{12}} \partial_{a_{21}} +(e^{g_{s}}-1) Q a_{12} \\
&\qquad + (e^{g_{s}}-1) e^{\hat p} a_{12} \partial_{a_{12}} + \mathcal{O}(a^2),
\end{align*}
where $\mathcal{O}(a)$ represents order in the variables $a=(a_{12},a_{21})$. The factors $(e^{g_{s}}-1)$ in front of disks with additional positive punctures comes from the perturbation scheme and are related to the gluing analysis in the proof of Proposition \ref{prp:cancel}, see \cite{EN}. In close analogy with the calculation at the classical level, the operators $\partial_{a_{12}}$ and $\partial_{a_{21}}$ can be eliminated and we get an operator equation which after change of framing to make $x$ correspond to the longitude of $T$, i.e., 0-framing, becomes
\begin{align*}
\qAug_T&= 
e^{g_{s}} Q^3 e^{3\hat p}(Q-e^{-3g_{s}}e^{2\hat p})(Q-e^{-g_{s}}e^{\hat p})\cdot 1 \\
&\quad + e^{-5g_{s}/2}(Q-e^{-2g_{s}}e^{2\hat p})\left((e^{2g_{s}} e^{2\hat p}+e^{3g_{s}} e^{2\hat p}-e^{3g_{s}} e^{\hat p}+e^{4g_{s}})Q^2 \right.\\
&\qquad\qquad\qquad\qquad\qquad\qquad\; 
\left.-(e^{g_{s}}e^{3\hat p}+e^{3g_{s}}e^{2\hat p}+e^{g_{s}}e^{2\hat p})Q+e^{4\hat p}\right)\cdot e^{\hat x} \\
&\quad +(Q-e^{-g_{s}} e^{2\hat p})(e^{\hat p}-e^{g_{s}})\cdot e^{2\hat x},
\end{align*}
in agreement with the recursion relation of the colored HOMFLY-PT in \cite{Garoufalidis}.

\bibliographystyle{hplain}
\bibliography{myrefsqc}

\end{document}